\documentclass{amsart}

\usepackage[left=3cm,right=3cm,bottom=2cm,top=3cm]{geometry}
\usepackage{amsmath}
\usepackage{amssymb}
\usepackage{mathrsfs}
\usepackage[all]{xy}
\usepackage[dvips]{graphicx}

\newtheorem{lem}{Lemma}[section]
\newtheorem{prop}[lem]{Proposition}
\newtheorem{cor}[lem]{Corollary}
\newtheorem{theorem}[lem]{Theorem}

\def\A{\mathcal{A}}
\def\C{\mathcal{C}}
\def\D{\mathcal{D}}
\def\G{\mathcal{G}}

\def\K{\mathcal{K}}

\def\T{\mathcal{T}}
\def\W{\mathcal{W}}
\def\DD{\mathscr{D}}
\def\TT{\mathscr{T}}
\def\ZZZ{\mathbb{Z}}
\def\ol{\overline}

\def\Hom{\operatorname{Hom}}
\def\End{\operatorname{End}}
\def\Ext{\operatorname{Ext}}
\def\add{\operatorname{add}}
\def\ind{\operatorname{ind}}
\def\ql{\operatorname{ql}}

\title{Cluster structures from 2-Calabi-Yau categories with loops}

\author{Aslak Bakke Buan}
\address{Institutt for matematiske fag\\
Norges teknisk-naturvitenskapelige universitet\\
N-7491 Trondheim\\
Norway}
\email{aslakb@math.ntnu.no}

\author{Robert J. Marsh}
\address{Department of Pure Mathematics\\
University of Leeds\\
Leeds LS2 9JT\\
England}
\email{marsh@maths.leeds.ac.uk}

\author{Dagfinn F. Vatne}
\address{Institutt for matematiske fag\\
Norges teknisk-naturvitenskapelige universitet\\
N-7491 Trondheim\\
Norway}
\email{dvatne@math.ntnu.no}

\begin{document}

\maketitle

\begin{abstract}
We generalise the notion of \emph{cluster structures} from the work of
Buan-Iyama-Reiten-Scott to include situations where the endomorphism
rings of the clusters may have loops. We show that in a Hom-finite
2-Calabi-Yau category, the set of maximal rigid objects satisfies these
axioms whenever there are no 2-cycles in the quivers of their endomorphism
rings.

We apply this result to the cluster category of a tube, and show that this
category forms a good model for the combinatorics of a type $B$ cluster
algebra.
\end{abstract}

\section*{Introduction}
Since the introduction of cluster algebras by Fomin and Zelevinsky \cite{fz1},
relationships between such algebras and interesting topics in several branches
of mathematics have emerged.

The project of modelling cluster algebras in a representation theoretic
setting was initiated in \cite{mrz}. Inspired by this, cluster categories were
defined in \cite{bmrrt} to be certain orbit categories obtained from the
derived category of Hom-finite hereditary abelian categories. These categories
have been widely studied for the case when the initial hereditary category is
the category of finite dimensional representations of an acyclic quiver
$Q$. (When $Q$ is a quiver
with underlying graph $A_n$, the cluster category was independently defined in
\cite{ccs}.) It has then been shown that the indecomposable rigid objects are
in bijection with the cluster variables in the cluster algebra $\A_Q$
associated with the same quiver, and under this bijection the clusters
correspond to the maximal rigid objects (\cite{ck} based on \cite{cc},
see also \cite{bmrtck}). Moreover, by \cite{bmr}, the quiver of the
endomorphism ring of a maximal rigid object is the same as the quiver
of the corresponding cluster.

The last phenomenon also appears for maximal rigid modules in the stable module
category of preprojective algebras of simply laced Dynkin type \cite{gls},
which is another example of a Hom-finite 2-Calabi-Yau triangulated
category. Inspired by this and \cite{iy}, \cite{kr}, an axiomatic framework
for mutation in 2-CY categories
was defined in \cite{birs}. The essential features were considered
to be: the unique exchange of indecomposable summands; the fact that exchange
pairs were related by approximation triangles; and the fact that on the level
of endomorphism rings, exchange of indecomposable summands led to
Fomin-Zelevinsky quiver mutation on the Gabriel quivers. For the third of
these features to make sense, one must require that the endomorphism rings
have no loops or 2-cycles in their quivers. In \cite{birs} it was also shown
that the collection of maximal rigid objects in any Hom-finite 2-CY
triangulated category fulfils these axioms for cluster structures whenever
the quivers of their endomorphism rings do not have loops or 2-cycles.

The cluster structures from \cite{birs} have two limitations: Firstly,
there exist Hom-finite 2-CY triangulated categories where the endomorphism
rings of maximal rigid objects do have loops and 2-cycles in their
quivers. The unique exchange property holds also in these categories
\cite{iy}, but FZ quiver mutation does not make sense in this
setting. Secondly, the cluster algebras which can be modelled from the
cases studied in \cite{birs} are the ones defined from quivers (equivalently,
skew-symmetric matrices), while cluster algebras can be defined also from more
general matrices.

The aim of this paper is to extend the notion of cluster structures from
\cite{birs}. We show that the set of maximal rigid objects in a Hom-finite
2-CY triangulated category satisfies this new
definition of cluster structure regardless of whether the endomorphism rings
have loops or not. (We must however assume that the quivers do not have
2-cycles.) One effect of this is that we will also relax the second
limitation, since the cluster algebras which can be modelled in this new
setting but not in the setting of \cite{birs} are defined from matrices which
are not necessarily skew-symmetric.

While previous investigations of cluster structures have regarded the quiver
of the endomorphism ring as the essential combinatorial data, our approach is
to emphasize the exchange triangles instead. We collect the information from
the exchange triangles in a matrix and require for our cluster structure that
exchange of indecomposable objects leads to FZ matrix mutation of this
matrix. In the no-loop situation considered in \cite{birs}, the matrix
defined here is the same as the one describing the quiver of the endomorphism
ring, so our definition is an extension of the definition in \cite{birs}.

In the cluster category defined from a module category of a hereditary
algebra, the maximal rigid objects are the same as the \emph{cluster-tilting}
objects. In general, however, the cluster-tilting condition is stronger, and
there exist categories where the maximal rigid objects are not
cluster-tilting. Examples of this in a geometrical setting are given in
\cite{bikr}. In this paper we present a class of examples from a purely
representation-theoretical setting, namely the cluster categories $\C_{\T_n}$
defined from \emph{tubes} $\T_n$. Here, the tube $\T_n$ is the category of
nilpotent representations of a quiver with underlying graph $\tilde{A}_{n-1}$
and with cyclic orientation. The category $\C_{\T_n}$ has also been studied in
\cite{bkl}.

An interesting aspect of cluster categories from tubes is that the
endomorphism rings of the maximal rigid objects have quivers with loops, but
not 2-cycles. Thus they are covered by the definition of cluster structures in
this paper,
but not by the one in \cite{birs}. We will show that the set of rigid objects
in $\C_{\T_n}$ is a model of the exchange combinatorics of a type $B$ cluster
algebra. To this end, we give a bijection between the indecomposable rigid
objects in $\C_{\T_n}$ and the cluster variables in a type $B_{n-1}$
cluster algebra such that the maximal rigid objects correspond to the
clusters. This bijection is via the \emph{cyclohedron}, or Bott-Taubes
polytope \cite{bt}. Also, we will see that the matrix defined from the
exchange triangles associated to a maximal rigid object is the same as the one
belonging to the corresponding cluster in the cluster algebra.

The article is organised as follows. In Section
\ref{section:clusterstructures} we give the new definition of cluster
structures and show that the set of maximal rigid objects in a Hom-finite 2-CY
category satisfies this definition whenever there are no 2-cycles in the
quivers of their endomorphism rings. In Section \ref{section:maxrigidintubes} we
give a complete description of the maximal rigid objects in the cluster
category of a tube. Finally, in Section \ref{section:relationshiptypeB}, we
show that the cluster structure in a cluster tube forms a good model of the
combinatorics of a type $B$ cluster algebra.

\section{Cluster structures} \label{section:clusterstructures}
In this section we generalise the notion of \emph{cluster structures} from
\cite{birs} to include situations where the quivers of the clusters may have
loops. We then proceed to show that the set of maximal rigid objects in a
Hom-finite 2-Calabi-Yau category admits a cluster structure with this new
definition, under the assumption that the Gabriel quivers of their
endomorphism rings do not have 2-cycles.

Let $k$ be some field. By a Hom-finite $k$-category we will mean a category
$\C$ where $\Hom_{\C}(X,Y)$ is a finite dimensional $k$-vector space for all
pairs of indecomposable objects $X$ and $Y$. We will normally suppress the
field $k$. A triangulated category $\C$ is said to be Calabi-Yau of dimension
2, or 2-CY, if $D\Ext_{\C}^i(X,Y)\simeq \Ext_{\C}^{2-i}(Y,X)$ for all objects
$X,Y$ of $\C$ and all $i$ in $\ZZZ$. For the remainder of this section, $\C$
will denote a Hom-finite 2-Calabi-Yau triangulated category.

We now recall the definition of a \emph{weak cluster structure} from
\cite{birs}. Let $\TT$ be a collection of sets of non-isomorphic indecomposable
objects of $\C$. Each set $T$ of indecomposables which is an element of $\TT$
is called a \emph{precluster}. The collection $\TT$ is said to have a
\emph{weak cluster structure} if the following two conditions are met:
\begin{itemize}
\item[(a)] For each precluster $T=\ol{T} \overset{\cdot}{\cup} \{ M\}$ with
  $M$ indecomposable, there exists a unique indecomposable $M^*\not \simeq
  M$ such that the disjoint union $T^*=\ol{T}\overset{\cdot}{\cup} \{ M^*\}$
  is a precluster. 
\item[(b)] $M$ and $M^*$ are related by triangles 
\[M^*\overset{f}{\to} U_{M,\ol{T}}\overset{g}{\to} M\to \quad \textrm{and}
\quad M\overset{s}{\to} U'_{M,\ol{T}}\overset{t}{\to} M^*\to\]
where $f$ and $s$ are minimal left $\add{\ol{T}}$-approximations and $g$ and
$t$ are minimal right $\add{\ol{T}}$-approximations. These triangles are
called the \emph{exchange triangles} of $M$ (equivalently, of $M^*$) with
respect to $\ol{T}$.
\end{itemize}
We will not distinguish between a precluster and the object obtained by taking
the direct sum of the indecomposable objects which form the precluster. The
same goes for subsets of preclusters.

Assume now that $\TT$ has a weak cluster structure. For each $T=\{ T_i\}_{i\in
  I}$ in $\TT$ we define a matrix $B_T=(b_{ij})$ by
\[
b_{ij}=\alpha_{U'_{T_i,T/T_i}}T_j-\alpha_{U_{T_i,T/T_i}}T_j
\]
where $\alpha_YX$ denotes the multiplicity of $X$ as a direct summand of
$Y$. So the matrix $B_T$ records, for each $T_i, T_j\in T$, the
difference between the multiplicities of $T_j$ in the two
exchange triangles for $T_i$. Note
that $b_{ii}=0$ for all $i$, since $T_i$ does not appear as a summand in the
target (resp. source) of a left (resp. right) $\add(T/T_i)$-approximation.

Recall Fomin-Zelevinsky \emph{matrix mutation}, defined in \cite{fz1}: For a
matrix $B=(b_{ij})$ the mutation at $k$ is given by $\mu_k(B)=(b^*_{ij})$, where
\[
b^*_{ij} = \left\{
\begin{array}{ll}
-b_{ij} & i=k\textrm{ or }j=k \\
b_{ij}+\frac{|b_{ik}|b_{kj}+b_{ik}|b_{kj}|}{2} & i,j\neq k
\end{array} \right.
\]
Note in particular that for $i,j\neq k$, we have $b_{ij}\neq b^*_{ij}$ if
and only if $b_{ik}$ and $b_{kj}$ are both positive or both negative.

We say that $\TT$ has a \emph{cluster structure} if $\TT$ has a weak cluster
structure, and in addition the following conditions are satisfied:
\begin{itemize}
\item[(c)] For each $T\in \TT$ and each $T_i\in T$, the objects $U_{T_i,T/T_i}$
  and $U'_{T_i,T/T_i}$ have no common direct summands.
\item[(d)] If $T_k$ is an indecomposable object and $T=\ol{T}\cup \{ T_k\}$
  and $T^*=\ol{T}\cup \{ T_k^*\}$ are preclusters, then $B_T$ and $B_{T^*}$
  are related by Fomin-Zelevinsky matrix mutation at $k$.
\end{itemize}
In this case, we call the elements of $\TT$ \emph{clusters}.

An interpretation of condition (c) above is that the endomorphism rings of
the clusters have Gabriel quivers which do not have 2-cycles. It should also
be noted that if there are no loops at the vertices corresponding to $T_i$ and
$T_j$ in these quivers, then the multiplicity of $T_j$ as a summand in the
exchange triangle for $T_i$ will equal the number of arrows between these two
vertices. In particular, if there are no loops at any of the vertices, $B_T$
is skew-symmetric and can be considered as a record of the quiver. Condition
(d) then reduces to FZ quiver mutation, so our definition coincides with the
definition in \cite{birs} in this case.

An object $T$ in a triangulated category $\K$ is said to be \emph{rigid} if
$\Ext_{\K}^1(T,T)=0$. It is called \emph{maximal rigid} if it is maximal with
this property, that is, $\Ext_{\K}^1(T\amalg X,T\amalg X)=0$ implies that
$X\in \add T$.

The collection of maximal rigid objects in any Hom-finite 2-CY triangulated
category has a weak cluster structure. This follows from \cite{iy}, as stated in
Theorem I.1.10 (a) of \cite{birs}. We can now prove a stronger version of
part (b) of the same theorem:

\begin{theorem} \label{thmmaxrigidcluster}
Let $\C$ be a Hom-finite 2-Calabi-Yau triangulated category, and let $\TT$ be
the collection of maximal rigid objects in $\C$. Assume $\TT$ satisfies
condition (c) above. Then $\TT$ has a cluster structure.
\end{theorem}

The proof of Theorem \ref{thmmaxrigidcluster} follows the same lines as the
proof of Theorem I.1.6. in \cite{birs}. We need to show that when an
indecomposable summand of a maximal rigid object is exchanged, the change in
the matrix is given by Fomin-Zelevinsky matrix mutation. The fact that the
matrices $B_T$ for maximal rigid $T$ are not necessarily skew-symmetric forces
us to prove different cases separately. For the proof we will need the
following lemma:

\begin{lem} \label{signskewsymm}
In the situation of the theorem, for any maximal rigid object $T$, $B_T$ is
sign skew symmetric, i.e., for all $i,j$, we have $b_{ij}<0$ if and only if
$b_{ji}>0$.
\end{lem}

\begin{proof}
As remarked after the definition of the matrix $B_T$, the diagonal entries
$b_{ii}$ vanish, so the statement in the lemma is clearly true for these.

Now assume $b_{ij}<0$ for some $i\neq j$. Then $T_j$ is a summand of the
middle term of the exchange triangle
\[
T^*_i \to U_{T_i,T/T_i}\to T_i\to
\]
Since the second map is a minimal right $\add(T/T_i)$-approximation,
this means that there exists a map $T_j\to T_i$ which does not factor through
any other object in $\add(T/(T_i\amalg T_j))$. In other words, there
is an arrow $T_j\to T_i$ in the quiver of $\End_{\C}(T)$. This in turn
implies that $T_i$ is a summand of the middle term in the exchange triangle
\[
T_j\to U'_{T_j,T/T_j}\to T^*_j\to
\]
and thus $b_{ji}>0$.
\end{proof}

\begin{proof}[Proof of Theorem \ref{thmmaxrigidcluster}]
Let $T=\amalg_{i=1}^{n}T_i$ be a maximal rigid object in $\C$. Suppose we want
to exchange the indecomposable summand $T_k$. Consider
$\ol{T}=\amalg_{i\neq k}T_i$ and the maximal rigid object $T^*=\ol{T}\amalg
T_k^*$. We want to show that the matrices $B_T=(b_{ij})$ and
$B_{T^*}=(b^*_{ij})$ are related by FZ matrix mutation. Pick two
indecomposable summands $T_i\not \simeq T_j$ of $\ol{T}$, so $i,j$ and $k$ are
all distinct.

For the purposes of this proof we introduce some notation. Let
\[
\alpha_a(T_b)=\alpha_{U_{T_a,T/T_a}}T_b
\]
that is, the multiplicity of $T_b$ as a direct summand of the middle term in
the exchange triangle ending in $T_a$ with respect to $T/T_a$. Similarly, we
denote the multiplicity of $T_b$ in the other triangle by
\[
\alpha'_a(T_b)=\alpha_{U'_{T_a,T/T_a}}T_b
\]
Note that under the assumption of condition (c), at least one of these two
numbers will be zero for any choice of $a,b$.

The exchange triangles for $T_k$ with respect to $\ol{T}$ are
\begin{equation} \label{exTk1}
T_k^* \to T_i^{\alpha_k(T_i)}\amalg T_j^{\alpha_k(T_j)}\amalg V_k \to T_k \to
\end{equation}
\begin{equation} \label{exTk2}
T_k \to T_i^{\alpha'_k(T_i)}\amalg T_j^{\alpha'_k(T_j)}\amalg V'_k \to T_k^* \to
\end{equation}
Note that $V_k$ and $V'_k$ do not have $T_i$ or $T_j$ as direct summand.
These are also the exchange triangles for $T_k^*$ with respect to
$\ol{T}$, but the roles of the middle terms are interchanged. It follows
immediately that $b^*_{ki}=-b_{ki}$, and the FZ formula holds for row $k$ of
the matrix.

In the rest of the proof we study the changes in row $i$, where $i\neq k$. For
this we also need the exchange triangles for $T_i$ with respect to $T/T_i$,
which are
\begin{equation} \label{exTi1}
T^*_i \overset{\phi_1}{\longrightarrow} T_j^{\alpha_i(T_j)}\amalg
T_k^{\alpha_i(T_k)}\amalg V_i \overset{\phi_2}{\longrightarrow} T_i \to
\end{equation}
\begin{equation} \label{exTi2}
T_i \overset{\phi_3}{\longrightarrow} T_j^{\alpha'_i(T_j)}\amalg
T_k^{\alpha'_i(T_k)}\amalg V'_i \overset{\phi_4}{\longrightarrow} T_i^* \to
\end{equation}
Again, note that $V_i$ and $V'_i$ do not have $T_j$ or $T_k$ as direct summand.
From these triangles we must collect information about the exchange triangles
of $T_i$ with respect to $\widetilde{T}=T/(T_i\amalg T_k) \amalg T^*_k$, since
these determine the entries in row $i$ of $B_{T^*}$.

We will consider three different cases, depending on whether $b_{ik}$ is
positive, negative or zero.

\textbf{Case I:} Assume $b_{ik}=0$, that is, $T_k$ does not
appear in any of the exchange triangles for $T_i$. Then, by Lemma
\ref{signskewsymm}, we have $b_{ki}=0$ as well, and by the above,
$b^*_{ki}=0$. By
appealing once more to Lemma \ref{signskewsymm}, we see that $T^*_k$ does not
appear in the exchange triangles for $T_i$ with respect to $\widetilde{T}$. This
is enough to establish that the map $\phi_1$ in triangle \eqref{exTi1} is also a
minimal left $\add \widetilde{T}$-approximation. Similarly, the map $\phi_4$ in
triangle \eqref{exTi2} is a minimal right $\add
\widetilde{T}$-approximation. This
means that the triangles \eqref{exTi1} and \eqref{exTi2} are also the exchange
triangles for $T_i$ with respect to $\widetilde{T}$. Thus the entries in row $i$
remain unchanged and behave according to the FZ rule in this situation. (Note
also that this proves that $T^*_i$ is the complement of $T_i$ both before and
after we have exchanged a summand $T_k$ with $b_{ik}=0$.)

\textbf{Case II:} Suppose now that $b_{ik}<0$, which means that $T_k$ appears
as a summand in \eqref{exTi1}, while $\alpha'_i(T_k)=0$. By Lemma
\ref{signskewsymm}, $b_{ki}>0$, which means that $T_i$ appears in
\eqref{exTk2}, not in \eqref{exTk1}. Our strategy is to
construct redundant versions of the exchange triangles of $T_i$ with respect to
$\widetilde{T}$. We will use the triangle
\begin{equation} \label{bigtriangle1}
\left( T^*_k\right)^{\alpha_i(T_k)}\longrightarrow
\begin{array}{c}
\left( T_j^{\alpha_i(T_j)}\amalg V_i\right) \\
\amalg \\
\left( T_j^{\alpha_k(T_j)}\amalg
  V_k\right)^{\alpha_i(T_k)} \\
\end{array}
\overset{\phi'_1}{\longrightarrow}
\begin{array}{c}
\left( T_j^{\alpha_i(T_j)}\amalg V_i\right) \\
\amalg \\
T_k^{\alpha_i(T_k)}
\end{array}
\longrightarrow
\end{equation}
which is the direct sum of $\alpha_i(T_k)$ copies of \eqref{exTk1} and the
identity map of $T_j^{\alpha_i(T_j)}\amalg V_i$. Applying the octahedral axiom to
the composition of the map $\phi'_1$ in \eqref{bigtriangle1} and the map
$\phi_2$ in \eqref{exTi1} yields the following commutative diagram in which
the middle two rows and middle two columns are triangles.
\[
\xymatrix{
&&\eqref{bigtriangle1}&& \\
&\left( T^*_k\right)^{\alpha_i(T_k)}[1]\ar@{=}[r]&\left(
  T^*_k\right)^{\alpha_i(T_k)}[1]&& \\
T_i[-1]\ar[r]\ar@{=}[d]&T^*_i\ar[r]\ar[u]^{\chi_1} &\left(
  T_j^{\alpha_i(T_j)}\amalg 
  V_i\right) \amalg T_k^{\alpha_i(T_k)} \ar[u]\ar[r]^-{\phi_2}
&T_i\ar@{=}[d]&\eqref{exTi1} \\
T_i[-1]\ar[r]&X\ar[u]\ar[r] &\left( T_j^{\alpha_i(T_j)}\amalg V_i\right)\amalg
\left( T_j^{\alpha_k(T_j)}\amalg
  V_k\right)^{\alpha_i(T_k)}\ar[u]^-{\phi'_1}\ar[r]^-{\phi}&T_i& \\
&\left( T^*_k\right)^{\alpha_i(T_k)}\ar[u] \ar@{=}[r] &\left(
  T^*_k\right)^{\alpha_i(T_k)}\ar[u]&&
}
\]
We now want to show that the map $\phi=\phi_2 \phi'_1$ is a (not
necessarily minimal) right $\add \widetilde{T}$-approximation. Any map $f:
T_t\to T_i$ where $t\neq i$ will factor through $\phi_2$ since this is a right
$\add (T/T_i)$-approximation, so $f=\phi_2 f_1$. But since $\phi'_1$ is a right
$\add(\ol{T})$-approximation, $f_1$ factors through $\phi'_1$. Thus $f$
factors through $\phi_2 \phi'_1=\phi$.

Suppose instead that we have a map $f:T^*_k\to T_i$. Let $h:T^*_k\to
T_j^{\alpha_k(T_j)}\amalg V_k$ be the minimal left $\add
\ol{T}$-approximation for $T^*_k$. Then $f=gh$ for some map
$g:T_j^{\alpha_k(T_j)}\amalg V_k \to
T_i$. Since $T_i$ is not a summand of $V_k$, we have that
$T_j^{\alpha_k(T_j)}\amalg V_k$ is in $\add (\widetilde{T}/T^*_k)$, and by the
above, $g$ factors through $\phi$. We conclude that $\phi$ is a right $\add
\widetilde{T}$-approximation.

Similarly we now construct a second commutative diagram. We use the octahedral
axiom on the composition of the map $\phi_4$ in \eqref{exTi2}
and the map $\chi_1$ in the triangle
\begin{equation}
\left( T^*_k\right)^{\alpha_i(T_k)} \longrightarrow X \longrightarrow T_i^*
\overset{\chi_1}{\longrightarrow} \left( T^*_k\right)^{\alpha_i(T_k)}[1]
\end{equation}
from the second column of the previous diagram. (Note that by our assumption,
$T_k$ does not appear in \eqref{exTi2}.)
\[
\xymatrix{
&\left( T^*_k\right)^{\alpha_i(T_k)}[1]\ar@{=}[r]&\left(
  T^*_k\right)^{\alpha_i(T_k)}[1]& \\
T_i\ar@{=}[d]\ar[r] &T_j^{\alpha'_i(T_j)}\amalg V'_i\ar[u]^{\chi}\ar[r]^{\phi_4}
&T^*_i\ar[u]^{\chi_1}\ar[r] &T_i[1]\ar@{=}[d] \\
T_i\ar[r]&Y\ar[u]\ar[r]^{\psi}&X\ar[u]\ar[r]^{\psi_1}&T_i[1] \\
&\left( T^*_k\right)^{\alpha_i(T_k)}\ar[u]\ar@{=}[r] &\left(
  T^*_k\right)^{\alpha_i(T_k)}\ar[u]& 
}
\]
We notice that in this second diagram, the map $\chi$ must be zero, since
$\Ext^1_{\C}(T_j,T^*_k)=\Ext^1_{\C}(V_i',T^*_k)=0$. Therefore, the triangle
splits, and
\[
Y\simeq \left( T_j^{\alpha'_i(T_j)}\amalg V'_i\right) \amalg \left(
  T^*_k\right)^{\alpha_i(T_k)}
\]
We see that $\psi$ in the diagram is a (not necessarily minimal) right $\add
\widetilde{T}$-approximation as well: For any map $f:U\to X$ where $U\in
\add \widetilde{T}$, the composition $\psi_1f$ is zero since
$\Ext^1_{\C}(U,T_i)=0$, which again implies that $f$ factors through $\psi$.

Since $\phi$ is a right $\add \widetilde{T}$-approximation, $X=Z\amalg T^!_i$,
where $T^!_i$ is the second complement of $\widetilde{T}$. Consider the
triangles we have constructed:
\begin{equation} \label{redundant1}
Z\amalg T^!_i
\longrightarrow
\left( T_j^{\alpha_i(T_j)}\amalg V_i\right)\amalg \left(
  T_j^{\alpha_k(T_j)}\amalg V_k\right)^{\alpha_i(T_k)}
\overset{\phi}{\longrightarrow}
T_i
\longrightarrow
\end{equation}
and
\begin{equation} \label{redundant2}
T_i
\longrightarrow
\left( T_j^{\alpha'_i(T_j)}\amalg V'_i\right) \amalg \left(
  T^*_k\right)^{\alpha_i(T_k)}
\overset{\psi}{\longrightarrow}
Z\amalg T^!_i
\longrightarrow
\end{equation}
Since $\phi$ and $\psi$ are right $\add(\widetilde{T})$-approximations, we see
that an automorphism of $Z$ splits off in both
triangles, and the remaining parts are the exchange triangles for $T_i$ and
$T^!_i$ with respect to $\widetilde{T}$. So to find the entry $b^*_{ij}$ we
calculate the difference of the multiplicities of $T_j$ in the two triangles
\eqref{redundant1} and \eqref{redundant2}. So if we denote by
$\alpha_{\eqref{redundant1}}(T_j)$ the multiplicity of $T_j$ in the middle term
of triangle \eqref{redundant1} and similarly for triangle \eqref{redundant2}
we get
\begin{eqnarray*}
b^*_{ij} & = & \alpha_{\eqref{redundant2}}(T_j)-\alpha_{\eqref{redundant1}}(T_j) \\
& = & \alpha'_i(T_j)-\left( \alpha_i(T_j)+\alpha_k(T_j)\alpha_i(T_k)\right) \\
& = & \left( \alpha'_i(T_j)-\alpha_i(T_j)\right)-\alpha_k(T_j)\alpha_i(T_k) \\
& = & \left\{ \begin{array}{ll} b_{ij} & \textrm{ when
    }\alpha_k(T_j)=0\textrm{, i.e.\ when }b_{kj}\geq 0 \\
    b_{ij}-(-b_{kj})(-b_{ik})=b_{ij}-b_{kj}b_{ik} & \textrm{ when
    }\alpha_k(T_j)>0\textrm{, i.e.\ when }b_{kj}<0 \end{array}\right.
\end{eqnarray*}
Also, it is clear from \eqref{redundant1} and \eqref{redundant2} that
$b^*_{ik}=\alpha_i(T_k)=-b_{ik}$. Summarising, we see that the entries in the
$i$th row change as required by the FZ rule in this case.

\textbf{Case III:} Finally we consider the case where $b_{ik}>0$. This means
that $T_k$ appears in \eqref{exTi2}, but not in \eqref{exTi1}. Furthermore,
by Lemma \ref{signskewsymm}, $T_i$ appears in \eqref{exTk1}, but not in
\eqref{exTk2}. The argument follows the same lines as in Case II. Instead of
\eqref{bigtriangle1}, we use the following triangle:
\begin{equation} \label{bigtriangle2}
\left( T^*_k\right) ^{\alpha'_i(T_k)}[-1]\longrightarrow
\begin{array}{c}
\left( T_j^{\alpha'_i(T_j)}\amalg V'_i\right) \\
\amalg \\
T_k^{\alpha'_i(T_k)}
\end{array}
\overset{\phi'_2}{\longrightarrow}
\begin{array}{c}
\left( T_j^{\alpha'_i(T_j)}\amalg V'_i\right) \\
\amalg \\
\left( T_j^{\alpha'_k(T_j)}\amalg V'_k\right)^{\alpha'_i(T_k)}
\end{array}
\longrightarrow
\left( T^*_k\right)^{\alpha'_i(T_k)}
\end{equation}
which is the direct sum of $\alpha'_i(T_k)$ copies of \eqref{exTk2} and the
identity map of $T_j^{\alpha'_i(T_j)}\amalg V'_i$. By the octahedral axiom,
applied to the composition of the map $\phi_3$ in \eqref{exTi2} and the map
$\phi'_2$ in \eqref{bigtriangle2}, we get this commutative diagram:
\[
\xymatrix{
&\left( T^*_k\right)^{\alpha'_i(T_k)}\ar@{=}[r] &\left(
  T^*_k\right)^{\alpha'_i(T_k)}& \\
T_i\ar[r]^-{\phi'} &\left( T_j^{\alpha'_i(T_j)}\amalg V'_i\right) \amalg \left(
  T_j^{\alpha'_k(T_j)}\amalg V'_k\right)^{\alpha'_i(T_k)}\ar[u]\ar[r] &
X\ar[u]\ar[r]&T_i[1] \\
T_i\ar@{=}[u]\ar[r]^-{\phi_3} &\left( T_j^{\alpha'_i(T_j)}\amalg V'_i\right)\amalg
T_k^{\alpha'_i(T_k)}\ar[u]^{\phi'_2}\ar[r]
&T^*_i\ar[u]\ar[r] &T_i[1]\ar@{=}[u] \\
&\left( T^*_k\right) ^{\alpha'_i(T_k)}[-1]\ar[u]\ar@{=}[r] &\left( T^*_k\right)
^{\alpha'_i(T_k)}[-1]\ar[u]& 
}
\]
By arguments dual to those in Case II, we can check that $\phi'$ is a
left (not necessarily minimal) $\add \widetilde{T}$-approximation. Details are
left to the reader.

Now the octahedral axiom applied to the composition in the left square below
gives a new diagram and a new object $Y$, where the second column is the same
triangle as the third column in the previous diagram:
\[
\xymatrix{
&\left( T^*_k\right)^{\alpha'_i(T_k)}\ar@{=}[r] &\left(
  T^*_k\right)^{\alpha'_i(T_k)} & \\
T_i[-1]\ar[r] & X\ar[u]\ar[r]^{\psi'} & Y\ar[u]\ar[r] & T_i \\
T_i[-1]\ar@{=}[u]\ar[r] & T^*_i\ar[u]\ar[r] & T_j^{\alpha_i(T_j)}\amalg
V_i\ar[u]\ar[r] & T_i\ar@{=}[u] \\
&\left( T^*_k\right) ^{\alpha'_i(T_k)}[-1]\ar[u]\ar@{=}[r]&\left( T^*_k\right)
^{\alpha'_i(T_k)}[-1]\ar[u]^{\chi'}&
}
\]
Once again, by the fact that $T_k$ is not a summand of $V_i$, and the
vanishing of $\Ext^1_{\C}$-groups, $\chi'$ in this diagram is zero, the triangle
splits, and
\[
Y\simeq \left( T^*_k\right)^{\alpha'_i(T_k)}\amalg T_j^{\alpha_i(T_j)}\amalg V_i
\]
and we may also conclude that $\psi'$ is a left $\add
\widetilde{T}$-approximation as in the previous case.

As in Case II, we can find the entry $b^*_{ij}$ by subtracting the
multiplicity of $T_j$ in the triangle involving $\psi'$ from its multiplicity
in the triangle involving $\phi'$:
\begin{eqnarray*}
b^*_{ij} & = & \left( \alpha'_i(T_j)+\alpha'_k(T_j)\alpha'_i(T_k)\right) -
\alpha_i(T_j) \\
& = & \left( \alpha'_i(T_j)-\alpha_i(T_j)\right) +
\alpha'_k(T_j)\alpha'_i(T_k) \\
& = & \left\{
\begin{array}{ll}
b_{ij} & \textrm{when }\alpha'_k(T_j)=0\textrm{, i.e.\ when }b_{kj}\leq 0 \\
b_{ij}+b_{kj}b_{ik} & \textrm{when }\alpha'_k(T_j)>0\textrm{, i.e.\
  when }b_{kj}>0
\end{array}
\right.
\end{eqnarray*}
Also, arguing in a similar way to Case II, we obtain
$b^*_{ik}=-\alpha'_i(T_k)=-b_{ik}$. We have shown that $b^*_{ij}$ is obtained
from $b_{ij}$ using the FZ mutation rule, so the proof is complete.
\end{proof}

\section{Maximal rigid objects in cluster categories of
  tubes} \label{section:maxrigidintubes}
In this section we will give a complete description of the maximal rigid
objects in the cluster category of a tube, as defined in \cite{bmrrt}. It
turns out that none of these are cluster-tilting objects. In Section
\ref{section:relationshiptypeB}, we will apply the main result in Section
\ref{section:clusterstructures} to show that this category provides a model
for the combinatorics of a type $B$ cluster algebra.

We will denote by $\T_n$ the \emph{tube of rank $n$}. One realization of this
category is as the category of nilpotent representations of a quiver with
underlying graph $\tilde{A}_{n-1}$ and cyclic orientation. We will write just
$\T$ for this category if the actual value of $n$ is not important. The
category $\T$ is a Hom-finite hereditary abelian category, and we can
therefore apply the definition from \cite{bmrrt} to form its cluster category.

The AR-quiver of the bounded derived category $\D^b(\T)$ of $\T$ is a
countable collection of copies of the tube, one for each shift. See Figure
\ref{fig:derivedtube}. The only maps in $\D^b(\T)$ which are not visible as
a composition of finitely many maps in the AR-quiver are the maps from each
$\T[i]$ to $\T[i+1]$ which correspond to the extensions in $\T$.
\begin{figure}
\centering
\includegraphics[width=9cm]{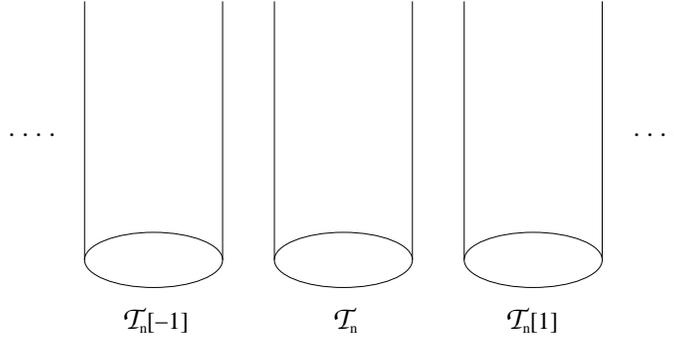}
\caption{The AR-quiver of the derived category of $\T_n$; a countable set of
  disconnected tubes. There exist maps from indecomposables in each copy to
  the next copy on the right, corresponding to extensions in the tube
  itself.} \label{fig:derivedtube}
\end{figure}

The cluster category is now defined as the orbit category
$\C_{\T_n}=\D^b(\T_n)/\tau^{-1} [1]$ where $\tau$ is the AR translation and
$[1]$ is the shift functor. Again, we will sometimes write just $\C_{\T}$.
There is a 1-1 correspondence between
the indecomposable objects in $\C_{\T_n}$ and those of $\T_n$, since $\ind \T_n$
is itself a fundamental domain for the action of $\tau^{-1}[1]$. We will
denote both an object in $\T_n$ and its orbit as an object in $\C_{\T_n}$ by the
same symbol, and we will sometimes refer to the category $\C_{\T_n}$ as a
\emph{cluster tube}.

Since $\T$ does not have tilting objects, it does not follow directly from
Keller's theorem \cite{k} that $\C_{\T}$ is triangulated. However, $\C_{\T}$
can be shown to be a thick subcategory of $\C_{H}$, the cluster category of a
suitable tame hereditary algebra $H$, or, as in \cite{bkl}, a subcategory of the
category of sheaves over a weighted projective line. It follows that $\C_{\T}$
is triangulated, and that the canonical functor $\D^b(\T_n)\to C_{\T_n}$ is a
triangle functor.

We will use a coordinate system on the indecomposable objects. We will
let $(a,b)$ be the unique object with socle $(a,1)$ and quasi-length
$b$, where the simples are arranged such that $\tau
(a,1)=(a-1,1)$ for $1\leq a\leq n$. Throughout, when we write equations and
inequalities which involve first coordinates outside the domain $1,...,n$, we
will implicitly assume identification modulo $n$. See Figure
\ref{fig:coordinates}.
\begin{figure}
\centering
\includegraphics[width=11cm]{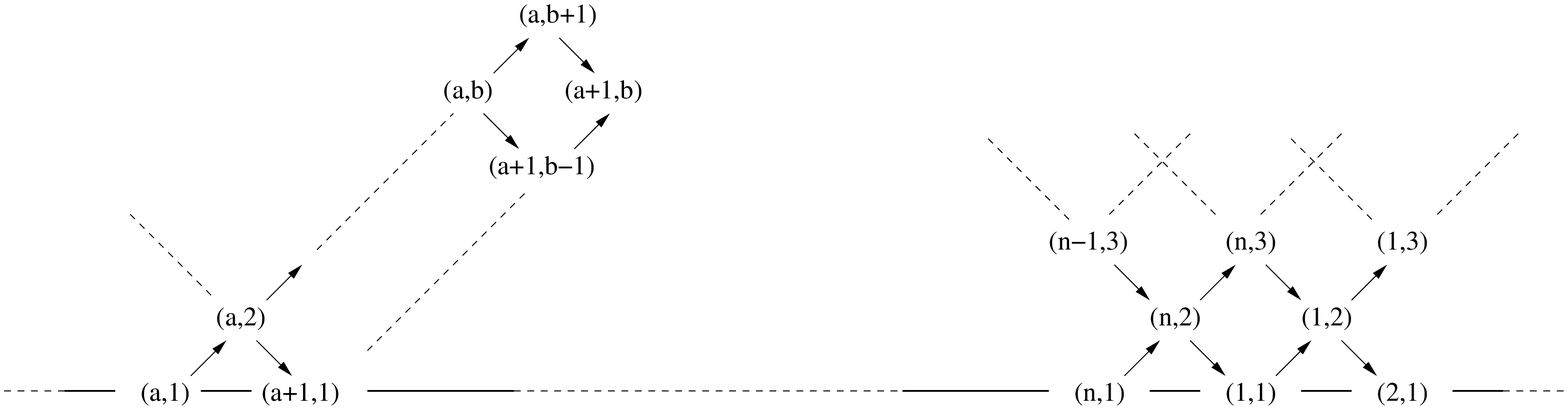}
\caption{Coordinate system for the indecomposable objects in the
  tube.} \label{fig:coordinates}
\end{figure}

\begin{lem} \label{lem:Homspaces}
If $X$ and $Y$ are indecomposables in $\T$, we have
\[
\Hom_{\C_{\T}}(X,Y) \simeq D\Hom_{\T}(Y,\tau^2 X) \amalg \Hom_{\T}(X,Y)
\]
where $D$ denotes the $k$-vector space duality $\Hom_k(-,k)$.
\end{lem}

\begin{proof}
By the definition of orbit categories,
\[
\Hom_{\C_{\T}}(X,Y) = \coprod_{i\in \mathbb{Z}}\Hom_{\D^b(\T)}(\tau^{-i}X[i],Y)
\]
Since $\T$ is hereditary, the only possible contribution can be for
$i=-1,0$:
\begin{eqnarray*}
\Hom_{\C_{\T}}(X,Y) & = &  \Hom_{\D^b(\T)}(\tau X[-1],Y) \amalg
\Hom_{\D^b(\T)}(X,Y) \\
& \simeq & \Hom_{\D^b(\T)}(\tau X,Y[1]) \amalg \Hom_{\D^b(\T)}(X,Y) \\
& = & \Ext^1_{\T}(\tau X,Y) \amalg \Hom_{\T}(X,Y) \\
& \simeq & D\Hom_{\T}(Y,\tau^2 X) \amalg \Hom_{\T}(X,Y)
\end{eqnarray*}
\end{proof}
\begin{figure}
\centering
\includegraphics[width=10cm]{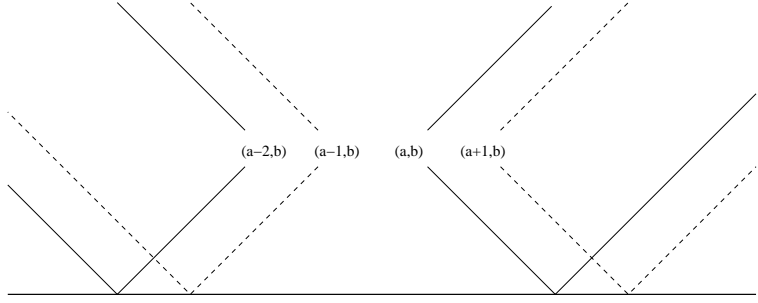}
\caption{For an indecomposable $X=(a,b)$, the Hom-hammock is illustrated by
  the full lines. Shifting it one to the right, we get the Ext-hammock. The
  backwards and forwards hammocks will overlap, depending on the rank of the
  tube.} \label{fig:hammocks}
\end{figure}

The Hom- and Ext-hammocks of an indecomposable object $X$ (that is, the
supports of $\Hom_{\C_{\T}}(X,-)$ and $\Ext^1_{\C_{\T}}(X,-)$) are illustrated
in Figure \ref{fig:hammocks}. For two indecomposables $X$ and $Y$ in
$\C_{\T}$, the maps in $\Hom_{\C_{\T}}(X,Y)$ which are images of maps in $\T$
will be called $\T$-maps, while those that come from maps $\T \to \T[1]$ in
the derived category will be called $\D$-maps.

The following lemma is necessary for understanding endomorphism rings of
objects in $\C_{\T}$.

\begin{lem} \label{lem:Dfactoring}
Let $X$ and $Y$ be two indecomposable objects in $\C_{\T}$. A
$\D$-map in $\Hom_{\C_{\T}}(X,Y)$ factors through the ray starting in $X$ and
the coray in ending in $Y$.
\end{lem}

\begin{proof}
Let $X$ and $Y$ be objects in $\T$ which correspond to $X$ and $Y$ in
$\C_{T}$. Let $\widetilde{Y}=\tau^{-1}Y[1]$ in $\D^b(\T)$. We will prove that
a map $f\in \Hom_{\D^b(\T)}(X,\widetilde{Y})$ factors through the ray starting
in $X$ in the AR-quiver of $\D^b(\T)$. This will imply that the image of $f$
in $\C_{\T}$ factors through the ray starting in $X$ in the AR-quiver
of $\C_{\T}$.

Since $f:X\to \widetilde{Y}$ is not an isomorphism in $\D^b(\T)$ for any
choice of $X$ and $Y$, it is enough to show that $f$ factors through the
irreducible map $g$ which forms the start of the ray, and we can do this by
induction on the quasi-length of $X$. If $\ql X=1$, then $g$ is a left almost
split map in $\D^b(\T)$, and the claim holds.

Suppose now that $\ql X\geq 2$. There is an almost split triangle in $\D^b(\T)$
\[
\xymatrix@R0.3pc{
&&Z\ar[drr]^h&&& \\
X\ar[urr]^g\ar[drr]_{g'}&&\amalg&&\tau^{-1}X\ar[r]& \\
&&Z'\ar[urr]_{h'}&&& \\
}
\]
Since $f$ is not an isomorphism, $f$ factors through the left almost split map
$g\amalg g'$. We have that $\ql{Z'}=\ql{X}-1$, so by induction we can assume
that any map in $\Hom_{\D^b(\T)}(Z',\widetilde{Y})$ factors through $h'$. By
the mesh relation, $h'\circ g'$ factors through $g$. Thus $f$ factors through
$g$.

The other assertion is proved dually.
\end{proof}

From now on, suppose $T$ is a maximal rigid object in $\C_{\T_n}$. We will
give a complete description of the possible configurations of indecomposable
summands of $T$. As a consequence, we will find that there are no
cluster-tilting objects in $\C_{\T}$.

\begin{lem}
All summands $T'$ of $T$ must have $\ql{T'}\leq n-1$.
\end{lem}

\begin{proof}
If $\ql{T'}\geq n$, then $\Ext_{\C_{\T_n}}^1(T',T')\neq 0$, so $T'$ is not a
summand of a rigid object.
\end{proof}

In what follows, for any indecomposable $X=(a,b)$ in $\C_{\T}$, the \emph{wing
  determined by $X$} will mean the set of indecomposables whose position in
the AR-quiver is in the triangle which has $X$ on top, that is, $(a',b')$ such
that $a'\geq a$ and $a'+b'\leq a+b$. We will use the notation $\W_X$ for this.

\begin{lem} \label{lem:wing}
If $T_0$ is an indecomposable summand of $T$ of maximal quasi-length, then all
the summands of $T$ are in the wing determined by $T_0$.
\end{lem}

\begin{proof}
We assume, without loss of generality, that $T_0$ has coordinates
$T_0=(1,l_0)$. Let $T_1=(s,l_1)$ be a summand of $T$ which maximises the sum
$x+y$ with $(x,y)$ the coordinates of summands of $T$. Assume that $T_1$ is
not in $\W_{T_0}$. Since $\Ext^1_{\C_{\T}}(T_0,T_1)=0$, we have $l_0+2\leq s$ and
$s+l_1\leq n$. Therefore all summands of $T$ sit inside the wing $\W_X$ where
$X=(1,s+l_1-1)$. Furthermore, $X$ has no self-extensions, so
$\Ext^1_{\C_{\T}}(T\amalg X,T\amalg X)=0$. Hence $X$ is a summand of
$T$. Since the quasi-length of $X$ is at least the quasi-length of $T_0$, we
must have $X=T_0$, but this contradicts the fact that $T_1$ is not in
$\W_{T_0}$. We conclude that all summands are in $\W_{T_0}$. See Figure
\ref{fig:wing}.
\end{proof}

\begin{figure}
\centering
\includegraphics[width=8cm]{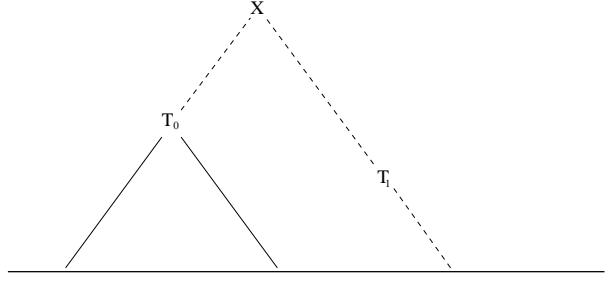}
\caption{In the proof of Lemma \ref{lem:wing}, if $T_1$ is outside the wing
  determined by $T_0$, then $T\amalg X$ is rigid. Note that $\ql X\leq n-1$,
  since $\Ext_{\C_{\T_n}}^1(T_0,T_1)=0$.} \label{fig:wing}
\end{figure}

For our maximal rigid object $T$ we will in the rest of this section denote by
$T_0$ the unique summand of maximal quasi-length, and we will sometimes call
it the top summand.

\begin{lem}
The quasi-length of $T_0$ is $n-1$.
\end{lem}

\begin{proof}
Suppose $\ql{T_0}=l_0<n-1$. Then $T_0=(m,l_0)$ for some $m\in
\{1,...,n\}$. The object $Y=(m,n-1)$ will satisfy $\Ext_{\C_{\T_n}}^1(T\amalg
Y,T\amalg Y)=0$. But this contradicts the fact that $T$ is maximal rigid.
\end{proof}

With the preceeding series of lemmas at our disposal, we get the following.

\begin{prop} \label{prop:tiltedbijection}
There is a natural bijection between the set of maximal rigid objects in
$\C_{\T_n}$ and the set
\[
\left\{ \textrm{tilting modules over }\vec{A}\right\} \times
\{1,...,n\}
\]
where $\vec{A}$ is a linearly oriented quiver of Dynkin type $A_{n-1}$.
\end{prop}

\begin{proof}
We have already established that all the summands of $T$ must be in the
wing determined by the summand $T_0$ of maximal quasi-length. This wing has
exactly the same shape as the AR-quiver of $k\vec{A}$. One can see easily that
for an indecomposable object in one such wing, the restriction of the
Ext-hammock to the wing exactly matches the (forwards and backwards)
Ext-hammocks of the corresponding indecomposable module in the AR-quiver of
$k\vec{A}$.

Thus the possible arrangements of pairwise orthogonal indecomposable objects
inside the wing match the possible arrangements of pairwise orthogonal
indecomposable modules in the AR-quiver of $k\vec{A}$.

Since we have $n$ choices for the top summand, we get the bijection by mapping
a maximal rigid object in the cluster tube to the pair consisting of the
corresponding tilting module over $k\vec{A}$ and the first coordinate of its
top summand.
\end{proof}

A rigid object $C$ in a triangulated 2-CY category $\C$ is called
\emph{cluster-tilting} if
$\Ext_{\C}^1(C,X)=0$ implies that $X\in \add C$. In particular, all
cluster tilting objects are maximal rigid. For cluster categories arising from
module categories of finite dimensional hereditary algebras, the opposite
implication is also true, namely that all maximal rigid objects are
cluster-tilting. The cluster tubes provide examples in which
this is not the case. 

\begin{cor} \label{noCTOs}
The category $\C_{\T}$ has no cluster-tilting objects.
\end{cor}

First a technical lemma:

\begin{lem}
Let $T$ be a maximal rigid object with top summand $T_0$, and $X$ an
indecomposable which is not in $\W_{T_0}$ and not in $\W_{\tau T_0}$. Then
$\Hom_{\C_{\T}}(T,X)=0$ if and only if $\Hom_{\C_{\T}}(T_0,X)=0$.
\end{lem}

\begin{proof}
For any indecomposable object $A$ and an indecomposable $B$ in the wing $\W_A$
determined by $A$, the restriction of the Hom-hammock of $B$ to $\ind \C_{\T}
\backslash \left( \W_A \cup \W_{\tau A} \right)$ is contained in the
Hom-hammock of $A$. (See Figure \ref{fig:hammocks}.)
\end{proof}

\begin{proof}[Proof of Corollary \ref{noCTOs}]
Let $T$ be a maximal rigid object and $T_0=(s,n-1)$ the top summand. Then the
objects $X_k=(s-1,kn-1)$, $k=1,2,3,...$, satisfy
$\Hom_{\C_{\T}}(T,X_k)=0$. (In fact, for $k>1$, these are the only objects
outside $\tau T$ on which $\Hom_{\C_{\T}}(T,-)$ vanishes.)

Now let $Y_k=\tau^{-1}X_k$. We then have
\[
\Ext_{\C_{\T}}^1(Y_k,T)=\Ext_{\C_{\T}}^1(T,Y_k)\simeq D\Hom_{\C_{\T}}(T,X_k)=0,
\]
but for $k>1$, $\Ext_{\C_{\T}}^1(T\amalg Y_k,T\amalg Y_k)\neq 0$, since $Y_k$
has self-extensions.
\end{proof}

As remarked after the definition of cluster structures in Section
\ref{section:clusterstructures}, an interpretation of condition (c) in the
definition is that the quiver of the endomorphism ring of the maximal rigid
object does not have 2-cycles. This is the case here:

\begin{prop} \label{prop:tubeclusterstructure}
The set of maximal rigid objects in $\C_{\T}$ has a cluster structure.
\end{prop}

\begin{proof}
By the fact that $\C_{\T}$ is a 2-Calabi-Yau triangulated category, and
Theorem \ref{thmmaxrigidcluster}, it is only necessary to show that for any
maximal rigid $T$ in $\C_{\T}$, there are no 2-cycles in the quiver of
$\End_{\C_{\T}}(T)$.

We note first that all the summands of $T$ lie in the wing $\W_{T_0}$, where
$T_0$ is the top summand of $T$. Suppose first that there is a non-zero
$\T$-map from $T_i$ to $T_j$, where $T_i$, $T_j$ are indecomposable direct
summands of $T$.

Since $\Ext^1_{\T}(T_j,T_i)=0$, it follows from the structure of $\T$ that there
is no non-zero $\T$-map from $T_i$ to $\tau^2 T_j$, and thus no non-zero
$\D$-map from $T_j$ to $T_i$. A non-zero $\T$-map from $T_j$ to $T_i$ cannot
arise from a path in the AR-quiver outside the wing (see Figure
\ref{fig:hammocks}), but it cannot arise from a path inside the wing either,
so there can be no non-zero $\T$-map from $T_j$ to $T_i$. Hence we see that
$\Hom_{\C_{\T}}(T_j,T_i)=0$.

Alternatively, suppose that there is a non-zero $\D$-map from $T_i$ to
$T_j$. By the above there can be no non-zero $\T$-map from $T_j$ to $T_i$. So
suppose that there is a non-zero $\D$-map from $T_j$ to $T_i$. Then we have
non-zero $\T$-maps from $T_i$ to $\tau^2 T_j$ and from $T_j$ to $\tau^2
T_i$. The only way for this to happen is for one of the summands
to be on the left hand edge of the wing and the other to be on the right hand
edge of the wing. Without loss of generality, assume $T_j$ is on the left edge
and $T_i$ is on the right edge. Then the non-zero $\D$-map from $T_j$ to $T_i$
factors through $T_0$ by Lemma \ref{lem:Dfactoring}, so there is no arrow from
$j$ to $i$.

Therefore, 2-cycles are not possible.
\end{proof}

One should notice that this set does not satisfy the definition of cluster
structures in \cite{birs}. The reason for this is that the quiver of
$\End_{\C_{T}}(T)$ will have a loop for all maximal rigid $T$. This is because
there is a non-zero $\D$-map from $T_0$ to itself. In addition, we have that
the only indecomposable objects in the wing of $T_0$ which $T_0$ has non-zero
maps to in $\C_{\T}$ are those on the right hand edge of the wing, and the
only indecomposable objects in the wing of $T_0$ which have non-zero maps in
$\C_{\T}$ to $T_0$ are those on the left hand edge of the wing. It follows
that the $\D$-map from $T_0$ to itself does not factor through any other
indecomposable object in the wing, and therefore it does not factor through
any other indecomposable direct summand of $T$.

\section{Relationship to type $B$ cluster
  algebras} \label{section:relationshiptypeB}
Cluster algebras were introduced in \cite{fz1}; see, for example, \cite{fr} or
\cite{fz4} for an introduction.

A simplicial complex was associated in \cite{fz3} to any finite root system, 
and it was conjectured there and later proved in \cite{cfz} that
these simplicial complexes are the face complexes of certain polytopes which
were called \emph{generalised associahedra}. In the finite type classification
of cluster algebras \cite{fz2}, it was shown that a cluster algebra is of
finite type if and only if its cluster complex is one of these simplicial
complexes.

The generalised associahedron associated to a type $B$ root system
turned out to be the \emph{cyclohedron}, also known as the Bott-Taubes
polytope \cite{bt}. This polytope was independently discovered by Simion
\cite{s}.

We will now recall the description of the exchange graph from \cite{fz3}
(which corresponds to the geometric description of the corresponding polytope
in \cite{s}).
Let $\G_n$ denote a regular $2n$-gon. The set of cluster variables in a type
$B_{n-1}$ cluster algebra is in bijection with the set $\DD_n$ of
\emph{centrally symmetric pairs of   diagonals} of $\G_n$, where the diameters
are included as degenerate pairs. Under this bijection, the clusters
correspond to the centrally symmetric triangulations of $\G_n$, and exchange
of a cluster variable corresponds to flipping either a pair of centrally
symmetric diagonals or a diameter.

In this section we define a bijection from the set of indecomposable rigid
objects in the cluster tube $\C_{\T_n}$ to the set $\DD_n$. This
map induces a correspondence between the maximal rigid objects and centrally
symmetric triangulations which is compatible with exchange. Thus rigid objects
in $\C_{\T_{n}}$ model the cluster combinatorics of type $B_{n-1}$ cluster
algebras.

We label the corners of $\G_n$ clockwise, say, from 1 to $2n$:
\[
\includegraphics[width=4cm]{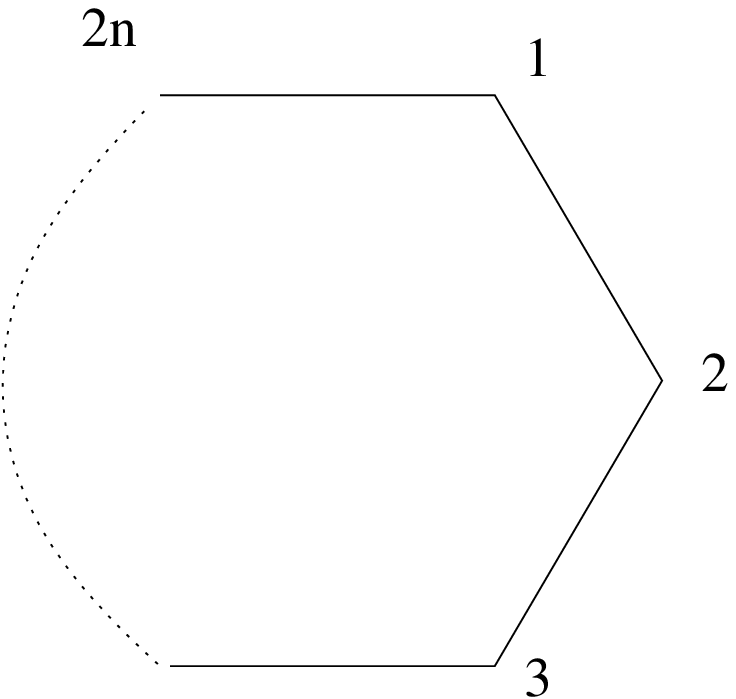}
\]
For two corners labelled $a$ and $b$ (which are neither equal nor neighbours),
we denote by $[a,b]$ the corresponding diagonal. Thus $[a,b]=[b,a]$, and
the centrally symmetric pairs of diagonals are given as $([a,b],[a+n,b+n])$.
(Here and in what follows, we reduce modulo $2n$ if necessary.) Since this
pair is uniquely determined by either of the two diagonals, we will sometimes
denote the pair by one of its representatives.

Now, for each indecomposable rigid object in $\C_{\T_n}$ we assign a
centrally symmetric pair of diagonals as follows:
\[
(a,b)\overset{\delta}{\longmapsto} ([a,a+b+1],[a+n,a+b+1+n])
\]
Note that the pairs of diagonals assigned to indecomposable objects in the
same ray or coray of $\C_{\T_n}$ all share a centrally symmetric pair of
corners. Objects of quasi-length 1 correspond to the shortest diagonals,
while the objects of quasi-length $n-1$ correspond to the diameters. See Figure
\ref{fig:B3triang}.

The following fact is readily verified.

\begin{lem} \label{lem:variablebijection}
The map $\delta$ defined above is a bijection from the set of indecomposable
rigid objects in $\C_{\T_n}$ to the set $\DD_n$ of centrally symmetric pairs of
diagonals of $\G_n$.
\end{lem}

\begin{figure}
\centering
\includegraphics[width=10cm]{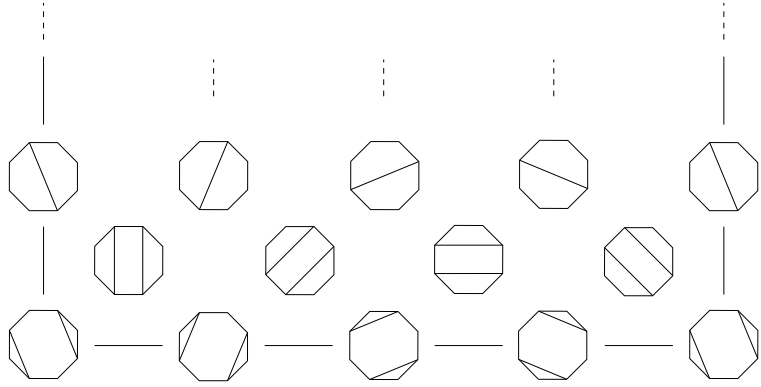}
\caption{The AR-quiver of $\C_{\T_{4}}$, with the indecomposable rigid objects
  replaced by their images under the map $\delta$.} \label{fig:B3triang}
\end{figure}

\begin{prop} \label{prop:crossing}
Let $T_1=(a,b)$ and $T_2=(c,d)$ be indecomposable rigid objects in
$\C_{\T_n}$. Then the number of crossing points of $\delta (a,b)$ and $\delta
(c,d)$ is equal to $2\dim \Ext^1_{\C_{\T}}(T_1,T_2)$.
\end{prop}

\begin{proof}
Without loss of generality, we may assume that $a=1$. We have
\begin{eqnarray*}
\delta(1,b) & = & ([1,b+2],[n+1,b+n+2]) \\
\delta(c,d) & = & ([c,c+d+1],[c+n,c+d+n+2]).
\end{eqnarray*}
It is easy to check that $\delta (1,b)$ and $\delta (c,d)$ cross if and only
if one of the following two conditions holds:
\begin{itemize}
\item[(i)] $1<c<b+2$ and $c+d>b+1$;
\item[(ii)] $1<c+d+1-n<b+2$ and $1<c<n+1$
\end{itemize}
Furthermore, the number of crossing points is 4 when both conditions hold, and
is 2 if only one holds. See Figure \ref{fig:crossings1}.
\begin{figure}
\centering
\includegraphics[width=13cm]{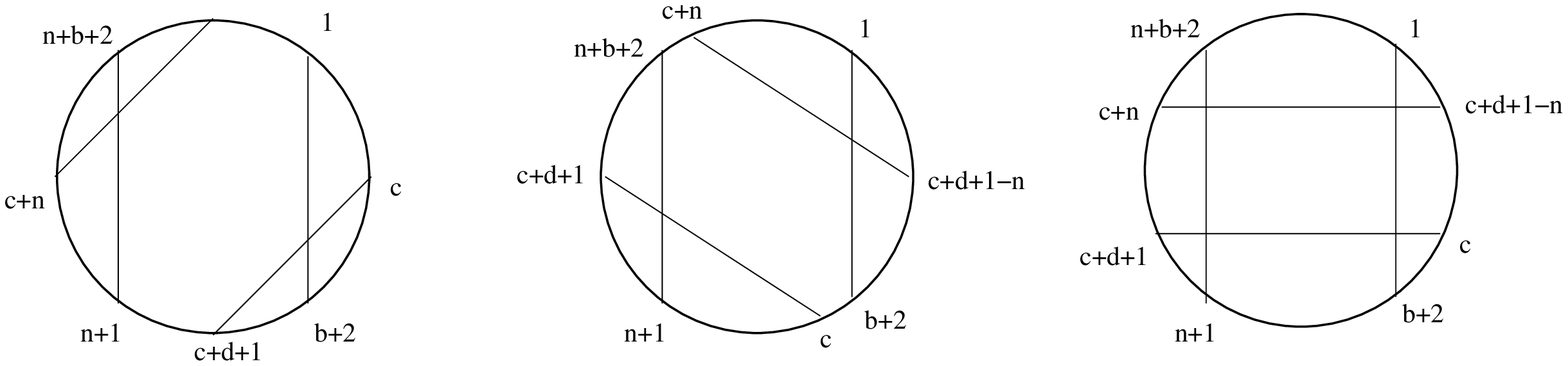}
\caption{Three cases in the proof of Proposition \ref{prop:crossing}. Left,
  condition (i) only holds; middle, condition (ii) only holds; right,
  conditions (i) and (ii) both hold.} \label{fig:crossings1}
\end{figure}

From the structure of the tube, it can be checked that condition (i) holds if
and only if $\Hom_{\T}(T_1,\tau T_2)\neq 0$, and in this case $\dim
\Hom_{\T}(T_1,\tau T_2)=1$. Similarly, condition (ii) holds if and only if
$\Hom_{\T}(T_2,\tau T_1)\neq 0$, and in this case $\dim \Hom_{\T}(T_2,\tau
T_1)=1$. By Lemma \ref{lem:Homspaces} and the Auslander-Reiten formula,
\[
\Ext^1_{\C_{\T}}(T_1,T_2)= D\Hom_{\T}(T_2,\tau T_1)\amalg \Hom_{\T}(T_1,\tau T_2)
\]
and the result follows.
\begin{figure}
\centering
\includegraphics[width=12cm]{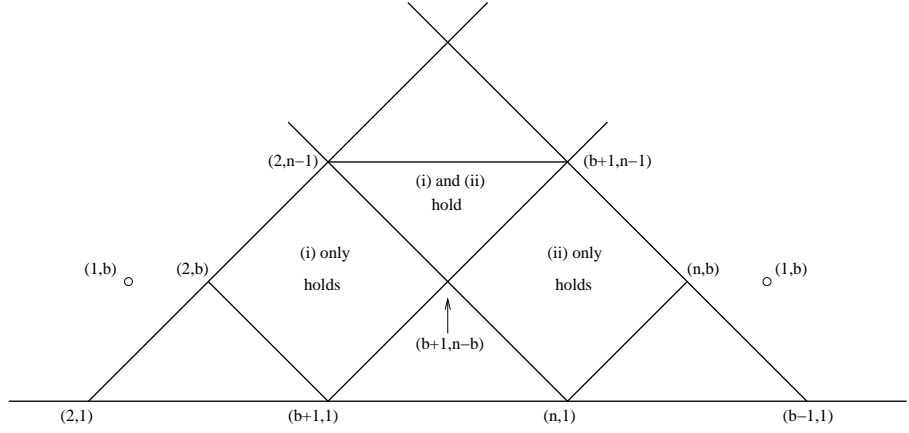}
\caption{Three situations for extensions with the object
  $(1,b)$, as in the proof of Proposition
  \ref{prop:crossing}} \label{fig:crossings2}
\end{figure}
\end{proof}

\begin{cor} \label{cor:tubeBbijection}
The map $\delta$ induces a bijection between the indecomposable rigid objects
in $\C_{\T_n}$ and the cluster variables in a type $B_{n-1}$ cluster algebra, and
under this bijection, the maximal rigid objects correspond to the clusters.
\end{cor}

\begin{proof}
By the work of \cite{fz2,fz3}, we need to show that the image under
$\delta$ of the summands of a maximal rigid object coincides with a set of
pairs of diagonals which form a centrally symmetric triangulation of
$\G_n$. This is clear from Lemma \ref{lem:variablebijection} and Proposition
\ref{prop:crossing}.
\end{proof}

Now let $T_{\textrm{init}}$ be the zig-zag maximal rigid object
\[
T_{\textrm{init}}=\coprod_{i=1}^t (i,n-2i+1) \amalg (i,n-2i)
\]
where $t=\frac{n}{2}$ for $n$ even and $t=\frac{n-1}{2}$ for $n$ odd, and
any expression with zero in the last coordinate is to be disregarded.

\begin{prop} \label{prop:typeBmatrix}
The Cartan counterpart of the matrix $B_{T_{\textrm{init}}}$ is the Cartan
matrix for the root system of type $B_{n-1}$.
\end{prop}

\begin{proof}
We recall that the Cartan counterpart of a matrix $B=(b_{ij})$ is the matrix
$A(B)=(a_{ij})$ given by $a_{ii}=2$ and $a_{ij}=-|b_{ij}|$ when $i\neq j$.

For this proof, we set $T_i$ to be the summand of $T_{\textrm{init}}$ with
quasi-length $n-i$, for $i=1,...,n-1$. In particular the top summand is $T_1$.

For convenience, we define $T_n=0$. For each $T_i$ with $i$ even and $2\leq
i\leq n-1$, the exchange triangles are
\[
T^*_i \to 0 \to T_i \to
\]
\[
T_i \to T_{i-1}\amalg T_{i+1} \to T^*_i \to
\]
while for $i$ odd and $1<i\leq n-1$, the exchange triangles are
\[
T^*_i \to T_{i-1}\amalg T_{i+1} \to T_i \to
\]
\[
T_i \to 0 \to T^*_i \to
\]
In the quiver of $\End_{\C_{\T}}(T_{\textrm{init}})$, there is a loop on the
vertex corresponding to the summand $T_1$. (See comment after Proposition
\ref{prop:tubeclusterstructure}.) However, twice around this loop is a zero
relation, so the exchange triangles for this summand are
\[
T^*_1 \to T_2\amalg T_2 \to T_1 \to
\]
\[
T_1 \to 0 \to T^*_1 \to
\]

So the matrix of the exchange triangles is
\[
B_{T_{\textrm{init}}}=
\left(
\begin{array}{cccccc}
0 & -2 & 0 & 0 & \cdots & 0 \\
1 & 0 & 1 & 0 & \cdots & 0 \\
0 & -1 & 0 & -1 & & 0 \\
0 & 0 & 1 & 0 & \ddots & \\
\vdots & \vdots & & \ddots & & (-1)^{n-2} \\
0 & 0 & & & (-1)^{n-1} & 0
\end{array}
\right)
\]
and the Cartan counterpart is
\[
A(B_{T_{\textrm{init}}})=
\left(
\begin{array}{cccccc}
2 & -2 & 0 & 0 & \cdots & 0 \\
-1 & 2 & -1 & 0 & \cdots & 0 \\
0 & -1 & 2 & -1 & & 0 \\
0 & 0 & -1 & 2 & \ddots & \\
\vdots & \vdots & & \ddots & & -1 \\
0 & 0 & & & -1 & 2
\end{array}
\right)
\]
\end{proof}

Noting that for a cluster algebra of type $B$, a cluster determines
its seed \cite{fz2}, we have thus proved the following theorem:

\begin{theorem}
There is a bijection between the indecomposable rigid objects of a cluster
tube of rank $n$ and the cluster variables of a cluster algebra of type
$B_{n-1}$, inducing a bijection between the maximal rigid objects of the
cluster tube and the clusters of the cluster algebra. Furthermore, the
exchange matrix of a seed coincides with the matrix associated to the
corresponding maximal rigid object in Section
\ref{section:clusterstructures}.
\end{theorem}

\begin{proof}
The first part of the theorem has been shown above (Corollary
\ref{cor:tubeBbijection}). The second part follows from Theorem
\ref{thmmaxrigidcluster} and Proposition
\ref{prop:typeBmatrix}, noting that the exchange matrix corresponding to the
initial root cluster $\{ -\alpha_1, -\alpha_2, \ldots, -\alpha_{n-1} \}$ of
type $B_{n-1}$ in \cite{fz2} is the matrix appearing in Proposition
\ref{prop:typeBmatrix} (see also Figure 5 in \cite{fz3}).
\end{proof}

\section*{Acknowledgements}
The first and third named author wish to thank Robert Marsh
and the School of Mathematics at the University of Leeds for their kind
hospitality in the spring of 2008.

The first author is supported by an NFR Storforsk-grant. The second author
acknowledges support from the EPSRC, grant number EP/C01040X/2, and is also
currently an EPSRC Leadership Fellow, grant number EP/G007497/1.

\end{document}